\documentclass{ws-rv9x6}
\usepackage{subfigure}   
\usepackage{ws-rv-thm}   
\usepackage{ws-rv-van}   
\usepackage{float}
\usepackage{appendix}

\makeindex

\newcommand{\cross}[1]{%
\mbox{\vbox{\kern 1pt\hbox{\vbox{\hrule
\kern 2pt\hbox{\ensuremath{\vphantom{b}#1}\kern 2pt}}\vrule\kern 1pt}}}\,}

\begin{document}

\chapter[The Q-Calculus: A Quaternion-Based
Laws of Form System]{The Q-Calculus: A Quaternion-Based\\
Laws of Form System\label{ra_ch1}}

\author[Louis H. Kauffman and Arthur M. Collings]{Louis H. Kauffman and Arthur M. Collings\footnote{Independent Researcher}}


\address{
loukau@gmail.com\footnote{Department of Mathematics, Statistices and Computer Science, University of Illinois at Chicago, 
Chicago, IL USA }\\
otter@mac.com\footnote{P.O. Box 114, Red Hook, NY 12571}}

\begin{abstract}
This paper introduces a Laws of Form version of the Quaternions. We call this the Q-Calculus, a 16-valued extension of Laws of Form (LoF) which is closely related to the BF Calculus (where we have a single square root of the mark) and the concept of the square root of negation (related to the the square root of minus one). We construct Q as a system of LoF mark operators acting on 4-tuples, and prove that the set of eight operators in Q is isomorphic to the quaternion group, which is non-commutative. We give a novel proof of several of Q's distribution laws using non-commutative logic gates. We indicate how to represent Q as braids by associating elementary braids to square roots of negation. This results in a very concise representation of Q as LoF braids. We end the paper with an indication of how we can represent the Artin braid group in LoF and how we can generalize our work with the quaternions to Clifford algebras.\\
\end{abstract}


\body


\section{Introduction}
We assume that the reader is generally familiar with  the quaternions and with Spencer-Brown\textquotesingle s \emph{Laws of Form}\cite{LOF1} (LoF). This paper describes $Q$, a  16-valued   extension to LoF that is closely related to our earlier development of the $BF$ Calculus \cite{BF3, BF4}. $Q$ is based on adding three \emph{imaginary} marks,\\ $\cross{\ \ }_i$\ , $\cross{\ \ }_j$\ , and $\cross{\ \ }_k$ to the LoF Calculus such that  
\begin{equation}
\cross{\cross{\ \ }_i}_i = \cross{\ \ }, \quad \cross{\cross{\ \ }_j}_j = \cross{\ \ }, \quad \cross{\cross{\ \ }_k}_k = \cross{\ \ } , \quad 
\label{eq0000}
\end{equation}
\begin{equation}
 \cross{\cross{\ \ }_i}_j = \cross{\ \ }_k\ \label{eq0001}
\end{equation} 
\begin{equation}
     \cross{\cross{\ \ }_\alpha\ }\ = \cross{\cross{\ \ }\ }_\alpha \ \ for\ \alpha \in \{i,j,k\}\ ,
     \label{eq0010}
\end{equation}
where \cross{\ \ } is \emph{LoF\textquotesingle s} orginal mark. Per Equation \ref{eq0010}, \cross{\ \ } commutes with the $i,j,$ and $k$ operators. Based on these assumptions, the resulting set of operators is isomorphic to the 8-element quaternion group $Q_8$. The non-commutativity of $\cross{\ \ }_i$ and $\ \cross{\ \ }_j$ is shown as follows. 
\begin{equation}\label{step1}
  \cross{\cross{\cross{\ \ }_i}_j}_k = \cross{\cross{\ \ }_k}_k = \cross{ \ \ }\ \ \ (\ref{eq0001})  \ (\ref{eq0000})
\end{equation}
\begin{equation}\label{step2}
  \cross{\cross{\ \ }_i} = \cross{\cross{\cross{\cross{\ \ }_i}_i}_j}_k = \cross{\cross{\cross{\ \ } }_j}_k =  \cross{\cross{\cross{\ \ }_j }_k} \ \ \ (\ref{step1}) \ (\ref{eq0000})\ (\ref{eq0010})
\end{equation}
\begin{equation}\label{step3}
 Therefore,\ (\ref{step2})\longrightarrow \ \cross{\ \ }_i = \cross{\cross{\ \ }_j }_k
\end{equation}
\begin{equation}\label{step4}
 \ \cross{\ \ }_j = \cross{\cross{\ \ }_k}_i\ \ \  \ We\ leave\ this \ for\ the \ reader. 
\end{equation}
$$(Hint: Show \ that\ \cross{\cross{\cross{\ }_j}_k}_i = \cross{\ }) $$
\begin{equation}\label{step5}
 \cross{\cross{\ \ }_k} = \cross{\cross{\cross{\ \ }_k}_i}_i = \cross{\cross{\ \ }_j}_i\ \ \ (\ref{eq0000})\ (\ref{step4})
\end{equation}
\begin{equation}\label{step6}
 Then\ (\ref{eq0001})\ and \   (\ref{step5})\longrightarrow \cross{\cross{\ \ }_i}_j\ne \cross{\cross{\ \ }_j}_i 
\end{equation}

Note that we are using the expressions as operators. See Section \ref{sec333} for a rigorous discussion of this. 

We are saying that we can directly add not only one new operator that is a square root of the mark, but more than one as in $i,j,k$ above. We can make an arithmetic for these new marks where the concatenations of them as operators is isomorphic with the quaternions. At this preliminary stage of developing the arithmetic (Section \ref{sec20}), we do not yet give any identities for juxtaposition. For example there is no reduction formula for $\cross{\ \ }_i$  $\cross{\ \ }_j$. In this way we can make a fundamental arithmetic that is a precursor to both the usual quaternions and to species of Laws of Form quaternions, where we let these operators act on four-tuples of marked and unmarked states. This gives a specific version of the quaternions in a Laws of Form context as we shall see in Section \ref{sec333}.\\

In making these constructions, we introduce non-commutativity into LoF. This may seem a radical departure to the reader associated with Laws of Form. In this regard we point out that non-commutativity was surely a surprise to Hamilton himself when he discovered the quaternions in 1843. In fact Hamilton is famous for having written on the Brougham bridge in Dublin the equations $$ii = jj = kk = ijk = -1$$ when he had found them in the course of walking near the bridge. Note that these equations do not seem to imply a non-commutative algebra, but, assuming products are associative, they imply, from $ijk = -1$, that $jk = i$ and $ij = k$, and from this that $-k = ji$ so that $ij = -ji.$ In this way Hamilton would show that the quaternions need be non-commutative and he produced the very first non-commutative algebra in mathematics. Just so, we inevitably find non-commutativity in this quaternionic extension of Laws of Form.\\

We will show in Section \ref{sec20}  how the full quaternion group follows from these assumptions. Note that from the point of view of \emph{Laws of Form}, we are starting with the idea of a square root of the mark as in the algebra BF\cite{BF3}, and we are considering what will happen if another square root of the mark is introduced. We then make a very simple assumption about the composition of these two operators, namely that the composition is also a square root of the mark and that the mark commutes with these operators. Remarkably, this produces the quaternionic structure.

In Section \ref{sec333}, we formally define the n-tuples that form the basis for the Q Calculus.  Section \ref{sec30} provides a brief review of the BF Calculus and the concept of the square root of negation. In Section \ref{sec40}, we focus on several of the 56 distribution laws in Q, and give demonstrations that are novel due to their incorporation of anti-commuting logic gates.  Finally, in Section \ref{sec50} we look to the diagrammatic structure of the quaternion operators and the closely related Klein four group. Representing the 4-tuples in Q as strands in the braid group provides the basis for generalizing Q to higher orders.

\section{Laws of Form}\label{sec20}

\subsection{The LoF Mark and -1}
At the simplest level, LoF describes a system, the Calculus of Indications, in which every expression is equivalent to one of two states: the marked state, which is indicated by \cross{\ \  }\ , or the unmarked state, which is indicated by the absence of the mark. In this sense, LoF can be considered a two-state (or two-valued) system. We indicate these states by calling (referencing) their names, for example by writing $\cross{\ \ }$ to refer to the marked side of the distinction. In the calculus, calling such a name twice is no different than calling it once, which gives rise to the \emph{Law of Calling}: 

\begin{equation}
   Law\ of\ Calling:\ \cross{\ \ }\ \cross{\ \ }= \cross{\ \ }
\end{equation}

 A key idea in the Calculus of Indications is the concept of \emph{crossing} the boundary of a distinction. A distinction is a space (or concept) that is divided into two areas, one of which is indicated by being marked, and the other by being unmarked. LoF's mark, \cross{ \ \ } , can be considered an instruction to cross from  its inside (which is empty) to its outside, which we consider to be marked. Based on the simple geometry of a distinction, the  \emph{Law of Crossing} then states: 
\begin{equation}
   Law\ of\ Crossing:\  \cross{\cross{\ \ }} = \quad \quad . 
\end{equation}

The Law of Crossing is closely related to the form of \emph{Reflexion}: 

\begin{equation}
    Reflexion: \quad \quad  \cross{\cross{X}} = X. 
\end{equation}

One way to understand the Law of Crossing and the form of Reflexion is to observe that the positive and negative integers form the sides of a distinction that is bounded by 0. Negative integers in fact are marked with a sign of negation while positive integers are unmarked. 

Consider the act of multiplying an integer $x$ by $-1$. As  anyone familiar with rudimentary arithmetic will know, multiplying an integer $x$ by $-1$ results in a product of opposite sign: negative if $x$ is positive, and positive if $x$ is negative. Effectively, each case is equivalent to issuing an instruction to \emph{cross over} to the opposite side of the boundary designated by 0, with each resulting  in a change of sign. For example, if $x = 5$, then (-1)(5) = -5, which changes the sign of $x$ from unmarked to marked, while if $x = -5$, then (-1)(-5) = 5, changing the sign from marked to unmarked. In general, the product
\begin{equation*}
    (-1)(-1)x = x
\end{equation*}
 closely corresponds with Crossing and Reflexion.\\

 \noindent
 \emph{Exponental Expressions as Operators}. In many cases, operators in LoF are represented based on a form of containment, such as $\cross{X}$. It is also  possible to express operators in other forms. Consider, for example, the form $a^b$ for $b \in \{ \cross{\ \ }, \cross{\cross{\ \ }} \}.$ If we take by definition $a^{\cross{\ }} = \cross{a}$ and $a^{\cross{\cross{\ }}} = a$, then $a^b = \cross{\cross{a}b}\ \cross{\cross{b}a} = Xor(a,b)$. 

\subsection{The quaternion Group $Q_8$}\label{sec22}

The eight quaternion operators consisting of $\{\pm1, \pm i, \pm j, \pm k\}$ form the non-commutative group $Q_8$ under multiplication, in accord with the following:

\begin{equation}i^2 = j^2 = k^2 = ijk = -1
\label{eqQQ}
\end{equation}
$$ij = k,\ jk = i,\ ki = j$$
$$ji = -k,\ kj = -i,\ ik = -j.$$

\noindent
Likewise, in Q the eight operators  form a group that is isomorphic to $Q_8$: 
$$\cross{X}_i,\ \ \cross{X}_j,\ \ \cross{X}_k, \ \ \cross{X},$$ 
$$\cross{\cross{X}_i},\ \ \cross{\cross{X}_j},\ \ \cross{X}_k, \ \ X. $$  
Here we write each operator as a function of $X$ where $X$ can be replaced by another operator with which it is composed. Note that in this notation $X$ alone means the identity operation.\\

Then identities $\ref{eq1011}$, $\ref{eq1100}$, and $\ref{eq1101}$    are sufficient to establish the quaternion relations (\ref{eqQQ}) for $Q_8$. 
\begin{equation}
\cross{\cross{X}_i}_i = \cross{X}, \quad \cross{\cross{X}_j}_j = \cross{X}, \quad \cross{\cross{X}_k}_k = \cross{X} , \quad 
\label{eq1011}
\end{equation}
\begin{equation}
 \cross{\cross{X}_i}_j = \cross{X}_k\ \label{eq1100}
\end{equation} 
\begin{equation}
     \cross{\cross{X}_\alpha\ }\ = \cross{\cross{X}\ }_\alpha \ \ for\ \alpha \in \{i,j,k\}\ ,
     \label{eq1101}
\end{equation}
\newpage
\emph{QR 1:} \\

$\cross{\cross{\cross{X}_i}_j}_k$\\[10pt]
\begin{tabular}{@{\hspace{7ex}}l l }
    $= \cross{\cross{X}_k}_k$ &[By Equation \ref{eq1100},\ $\cross{\cross{X}_i}_j = \cross{X}_k$ ]\\[10pt]
    $= \cross{X}$ &[By Equation \ref{eq1011} ]\\[10pt]
\end{tabular}

\noindent
\emph{QR 2:} \\[10pt]
$\cross{\cross{\cross{X}_j}_k}$\\[10pt]
\begin{tabular}{@{\hspace{7ex}}l l }
    $= \cross{\cross{\cross{X}}_j}_k$ &[By Equation \ref{eq1101},\ \cross{\ \ } commutes with $\cross{\ \ }_j$ and $\cross{\ \ }_k$]\\[10pt]
$= \cross{\cross{\cross{\cross{X}_i}_i}_j}_k$ &[By Equation \ref{eq1011},\ $\cross{\cross{X}_i}_i = \cross{X}$]\\[10pt]
$= \cross{\cross{X}_i}$ &[By QR 1,  $\cross{\cross{\cross{X}_i}_j}_k = \cross{X}$]\\[10pt]
$= \cross{\cross{\cross{X}_i}_j}_j$ &[By Equation \ref{eq1011},\ $\cross{\cross{X}_j}_j = \cross{X}$\\[10pt]
$= \cross{\cross{X}_k}_j$ &[By Equation \ref{eq1100},\ $\cross{\cross{X}_i}_j = \cross{X}_k$]\\[10pt]
\end{tabular}

\noindent
\emph{QR 3:} \\[10pt]
$\cross{\cross{X}_j}_k$\\[10pt]
\begin{tabular}{@{\hspace{7ex}}l l }
    $=\cross{\cross{\cross{\cross{X}_j}_k}}$ &[\cross{\cross{X}} = X]\\[10pt]
    $=\cross{\cross{\cross{X}_i}}$ &[QR 2]\\[10pt]
    $=\cross{X}_i$ &[\cross{\cross{X}} = X]\\[10pt]
\end{tabular}

We note that QR 2 and 3 establish that the $\cross{\ \ }_j$ and $\cross{\ \ }_k$ operations anti-commute. The remaining identities for $Q_8$ (\ref{eqQQ}) are established in similar fashion, and we leave their confirmation to the reader. 

\section{Formal 4-Tuple Operators}\label{sec333}

The quaternions are a hyper-complex algebra discovered by William Rowan Hamilton in 1843 that extends the concept of complex numbers to four dimensions. As will be familiar to most readers, quaternions are most frequently expressed in the form
\begin{equation}
   q = a + bi + cj + dk
\end{equation}
where $a,b,c,d$ are real numbers, and the operators $i$, $j$, and $k$ are defined such that $i^2 = j^2 = k^2 = ijk = -1$ and $ij = k$, $jk = i$, and $ki = j$. Notably, the quaternions do not commute, with $ij = -ji$, $jk = -jk$,  and $ki = - ik$. \\

\noindent
Now, consider multiplying $q$ on the right by $i$, $j$, and $k$. We will abstract the definition of the imaginary operators in $Q$ from these patterns. 

\begin{align*}
    qi &= (a + bi + cj + dk)\ i \\
    &= ai + bii + cji + dki\\
        &=ai + -b + -ck + dj \\
        &=-b + ai + dj + -ck
\end{align*}\\[-18pt]
\begin{align*}
    qj &= (a + bi + cj + dk)\ j \\
    &= aj + bij + cjj + dkj\\
        &=aj + bk + -c + -di \\
        &=-c + -di + aj + bk
\end{align*}\\[-18pt]
\begin{align*}
    qk &= (a + bi + cj + dk)\ k \\
    &= ak + bik + cjk + dkk\\
        &=ak + -bj + ci + -d \\
        &= -d + ci + -bj + ak
\end{align*}\\[-18pt]

\noindent
The reader should observe that negative signs and LoF marks correspond with Definition \ref{def1} on a tuple-by-tuple basis. The LoF mark is closely related to the numeric operation of multiplying by negative 1.\\

\begin{definition} The operators $\cross{\ \ }_i$\ , $\cross{\ \ }_j$\ , and $\cross{\ \ }_k$ acting on 4-tuples $X =(a,b,c,d)$ in $Q$ are defined as follows: \\
\begin{align}
    \cross{(a,b,c,d)}_i &= (\cross{b}, a, d, \cross{c})
    \label{eqiop}
\end{align}
\begin{align}
    \cross{(a,b,c,d)}_j &= (\cross{c}, \cross{d}, a, b)    \label{eqjop}
\end{align}
\begin{align}
    \cross{(a,b,c,d)}_k &= (\cross{d}, c, \cross{b}, a)    \label{eqkop}
\end{align}
\begin{align}
    \cross{(a,b,c,d)} &= (\cross{a}, \cross{b}, \cross{c}, \cross{d})    \label{eqcross}
\end{align}\\
where within each tuple, the operation $\cross{x}$ is defined in strict accord with its  definition in LoF.
\label{def1}
\end{definition}

Definition \ref{def1}  formally defines Q operators based on 4-tuples. It is straightforward to show that, acting on 4-tuples,  the $\cross{\ \ }_i$, $\cross{\ \ }_j$, and $\cross{\ \ }_k$ operators also form an 8-element group that is isomorphic to $Q_8$, as stated in the  following theorem. Note that per the containment relation, $\cross{\cross{ X}_i}_j$ means that $\cross{\ \ }_j$ $contains$  $\cross{X}_i$, while $\cross{\cross{ X}_j}_i$ means that $\cross{\ \ }_i$ $contains$  $\cross{X}_j$.
\begin{theorem}
    Acting on 4-tuples under the containment relation  defined in Definition \ref{def1}, the $\cross{\ \ }_i$, $\cross{\ \ }_j$, and $\cross{\ \ }_k$ operators generate an 8-element group that is isomorphic to $Q_8$. 
\end{theorem}
\begin{proof}
      The eight expressions  generated  by $\cross{\ \ }_i$, $\cross{\ \ }_j$, and $\cross{\ \ }_k$, form a closed set under containment operations. For $X = (a,b,c,d),$ the eight expressions consist  of  the functions of 
      $X = (a,b,c,d):$
      $$\cross{X},\ \cross{\cross{X}},\ \cross{X}_i, \ \cross{\cross{X}_i},\ \cross{X}_j,\  \cross{\cross{X}_j}\ 
      \cross{X}_k,\  \cross{\cross{X}_k}\ ,$$
where the eight  elements in $Q_8$ map to these functions as follows: 
    $$-1 \rightarrow \cross{X}, \ 1\rightarrow \cross{\cross{X}},\ i \rightarrow \cross{X}_i, \ -i \rightarrow \cross{\cross{X}_i},$$
    $$j \rightarrow \cross{X}_j, \ -j \rightarrow \cross{\cross{X}_j},\ k \rightarrow  \cross{X}_k, \ -k \rightarrow \cross{\cross{X}_k},$$

To prove this mapping is operation preserving, we show that the quaternion identities listed in (\ref{eqQQ}) are satisfied by the corresponding $Q$ expressions.\\ 

\noindent
QCC\ \ \ $\cross{\cross{(a,b,c,d)}}$ \quad \quad \quad \quad \quad  \\[10pt]
\begin{tabular}{@{\hspace{9ex}}l l }
    $= \cross{(\cross{a},\cross{b},\cross{c},\cross{d})}$ &[Definition \ref{def1}(\ref{eqcross})]\\[10pt]
    $= (\cross{\cross{a}},\cross{\cross{b}},\cross{\cross{c}},\cross{\cross{d}})$ &[Definition \ref{def1}(\ref{eqcross})]\\[10pt]
    $= (a,b,c,d)$ &[In each tuple, \cross{\cross{x}} = x in LoF]\\[16pt]
\end{tabular}

\noindent
QII.\ \ \ $\cross{\cross{(a,b,c,d)}_i}_i$\quad \quad \quad \quad $i^2 = -1$\\[10pt]
\begin{tabular}{@{\hspace{9ex}}l l }
    $= \cross{(\cross{b},a,d,\cross{c})}_i$ &[Definition \ref{def1}(\ref{eqiop})]\\[10pt]
    $=(\cross{a},\cross{b},\cross{c},\cross{d})$ &[Definition \ref{def1}(\ref{eqiop})]\\[10pt]
        $= \cross{(a,b,c,d)}$ &[Definition \ref{def1}(\ref{eqcross})]\\[16pt]
\end{tabular}

Demonstrations of QJJ and QKK, which correspond to $j^2 = -1$ and $k^2 = -1$ are similar and are left to the reader. \\

\noindent
QIJ\ \ \ $\cross{\cross{(a,b,c,d)}_i}_j$ \quad \quad \quad \quad $ij = k$\\[10pt]
\begin{tabular}{@{\hspace{9ex}}l l }
    $= \cross{(\cross{b},a,d,\cross{c})}_j$ &[Definition \ref{def1}(\ref{eqiop})]\\[10pt]
    $= (\cross{d},\cross{\cross{c}},\cross{b},a)$ &[Definition \ref{def1}(\ref{eqjop})]\\[10pt]
    $= (\cross{d},c,\cross{b},a)$ &[\ \cross{\cross{x}} = x for any x in LoF]\\[10pt]
    $= \cross{(a,b,c,d)}_k$ &[Definition \ref{def1}(\ref{eqkop})]\\[16pt]
\end{tabular}

\newpage
\noindent
QIJK\ \ \ $\cross{\cross{\cross{(a,b,c,d)}_i}_j}_k$ \quad \quad \quad \quad $ijk = -1$\\[10pt]
\begin{tabular}{@{\hspace{9ex}}l l }
    $= \cross{\cross{(a,b,c,d)}_k}_k$ &[QIJ]\\[10pt]
    $= \cross{(\cross{d},c,\cross{b},a)}_k$ &[Definition \ref{def1}(\ref{eqkop})]\\[10pt]
    $=(\cross{a},\cross{b},\cross{c},\cross{d})$ &[Definition \ref{def1}(\ref{eqkop})]\\[10pt]
    $= \cross{(a,b,c,d)}$ &[Definition \ref{def1}(\ref{eqcross})]\\[16pt]
\end{tabular}

QJI\ \ \ $\cross{\cross{(a,b,c,d)}_j}_i$\quad \quad \quad \quad $ji = -ij$\\[10pt]
\begin{tabular}{@{\hspace{9ex}}l l }
    $= \cross{(\cross{c}, \cross{d},a,b)}_i$ &[Definition \ref{def1}(\ref{eqjop})]\\[10pt]
    $= (\cross{\cross{d}}, \cross{c},b,\cross{a})$ &[Definition \ref{def1}(\ref{eqiop})]\\[10pt]
    $= (\cross{\cross{d}}, \cross{c},\cross{\cross{b}},\cross{a})$ &[\ \cross{\cross{x}} = x for any x in LoF]\\[10pt]
    $= \cross{(\cross{d}, c,\cross{b}, a)}$ &[\ Definition \ref{def1}(\ref{eqcross})]\\[10pt]
    $= \cross{\cross{(a,b,c,d)}_k}$ &[Definition \ref{def1}(\ref{eqkop})]\\[10pt]
        $= \cross{\cross{\cross{(a,b,c,d)}_i}_j}$ &[QIJ]\\[16pt]
\end{tabular}

\noindent
QMC.\ \ \ $\cross{\cross{(a,b,c,d)}_i}$\quad \quad \quad \quad (The LoF Mark commutes)\\[10pt]
\begin{tabular}{@{\hspace{9ex}}l l }
    $= \cross{(\cross{b},a,d,\cross{c})}$ &[Definition \ref{def1}(\ref{eqiop})]\\[10pt]
    $=(\cross{\cross{b}},\cross{a},\cross{d},\ \cross{\cross{c}})$ &[Definition \ref{def1}(\ref{eqcross})]\\[10pt]
    $= \cross{(
    \cross{a},\cross{b},\cross{c},\cross{d})}_i$ &[Definition \ref{def1}(\ref{eqiop})]\\[16pt]
    $= \cross{\cross{(a,b,c,d)}}_i$ &[Definition \ref{def1}(\ref{eqcross})]\\[16pt]
\end{tabular}

\noindent
QINV.\ \ \ $\cross{\cross{\cross{(a,b,c,d)}_\alpha}_\alpha}$\quad \quad \quad \quad (Showing $\cross{\ \ }_\alpha$ has an inverse)\\[10pt]
\begin{tabular}{@{\hspace{9ex}}l l }
    $= \cross{\cross{(a,b,c,d)}}$\quad \quad \quad \ \   &[$\cross{\cross{X}_\alpha}_\alpha = \cross{X}\ for\ \alpha = i, j, k$]\\[10pt]
    $= (a,b,c,d)$ &[QCC: $\cross{\cross{X}}= X$].\\[10pt]
\end{tabular}

We leave the remaining demonstrations to the reader.  
\end{proof}

\subsection{Juxtaposition}
In the preceding we have focused on describing enclosures made by the marks \cross{\ \ }, $\cross{\ \ }_i$, $\cross{\ \ }_j$, and $\cross{\ \ }_k$. We know define the \emph{juxtaposition} of any two expressions in $Q$. 

\begin{definition} \label{def2}
For any two expressions $X = (a,b,c,d)$ and $Y = (e,f,g,h)$, their juxtaposition $XY$ is defined on a tuple-wise basis, with
    \begin{equation}
        XY = (a,b,c,d)\ (e,f,g,h) = (ae, bf, cg, dh)
    \end{equation}
where $ae, bf, cg$, and $dh$ are the corresponding juxtapositions as defined in LoF within each of the 4-tuples. 
\end{definition}
Juxtapositions in  LoF generalize the \emph{Law of Calling} and the form of \emph{Integration (Int.)} in LoF. For example: 

\begin{example}\\[10pt]
\noindent
$(\cross{\ }, \cross{\ }, \ \ , \cross{\ } )\ (a,b,c,  \ \cross{\ })$\\[10pt]
\begin{tabular}{@{\hspace{7ex}}l l }
   $= (\cross{\ }\ a, \cross{\ } b,\ c , \ \cross{\ }\ \cross{\ } )$ &[By Definition \ref{def2}]\\[10pt]
    $= (\cross{\ },\ \cross{\ },\ c , \ \cross{\ }\ \cross{\ } )$ &[Applying Int. (\cross{\ }\ x = \cross{\ })\ \emph{within} LoF]\\[12pt]
    $= (\cross{\ },\ \cross{\ },\ c , \ \cross{\ } )$ &[Applying the Law of Calling \emph{within} LoF]\\[16pt]
\end{tabular}
\end{example}
\newpage
\begin{example}\\[10pt]
  \noindent
$\cross{(a,b,c,d) }_i\ \ \cross{(a,c,b,d)}_j$\\[10pt]
\begin{tabular}{@{\hspace{7ex}}l l }
   $ = (\cross{b},a,d,\cross{c}) \ \ \cross{(a,c,b,d)}_j$ &[Definition \ref{def1}(\ref{eqiop})]\\[10pt]
   $ = (\cross{b},a,d,\cross{c}) \ \ (\cross{b}, \cross{d}, a,c)$ &[Definition \ref{def1}(\ref{eqjop})]\\[10pt]
$ = (\cross{b} \cross{b},\ a\ \cross{d},\ ad,\ \cross{c} c)$ &[Definition \ref{def2}]\\[10pt] 
$ = (\cross{b},\ a\ \cross{d},\ ad,\ \cross{\ \ } )$ &[Simplifying each tuple in LoF]\\[10pt] 
\end{tabular}  
\end{example}

The preceding  gives us a basis for defining whether two expressions in $Q$ are equivalent. 

\begin{definition}
    Two expressions $A = (a_1, a_2, a_3, a_4)$ and $B = (b_1, b_2, b_3, b_4)$ are equivalent ($A = B$) if and only if for all $x$, $a_x = b_x$ \emph{as expressions in LoF}. 
\end{definition}

\begin{remark}
    
In LoF, two expressions $a_x$ and $b_x$ are equivalent if the following are true. Due to the LoF Completeness Theorem, if either (1) or (2) is true then both are true.  
\begin{enumerate}
    \item A comprehensive analysis of cases (such as a truth table) proves that $a_x = b_x$ in every  case. 
    \item Alternately, $a_x = b_x$ is demonstrable by applying LoF\textquotesingle s algebraic initials and consequences in a finite series of steps to obtain $b_x$ from $a_x$. 
\end{enumerate}

\end{remark}

\subsection{Notation}
The notation we have chosen in Q affords  considerable degree of flexibility for representing expressions. This is a good point to review some of these options,  beginning with treatment of the umarked state. 
\begin{itemize}
    \item Consider the unmarked state $U = (\ \ , \ \ , \ \ , \ \ )$. If the mark encloses $U$, then we have $\cross{(\ \ , \ \ , \ \ , \ \ )} = (\cross{\ }, \cross{\ }, \cross{\ }, \cross{\ })$. Frequently in the case of the unmarked state, we may choose to write $U = \ \ \ .$ Then in case of enclosure by the mark, we have $\cross{ \ \ \ } = (\cross{\ }, \cross{\ }, \cross{\ }, \cross{\ })$. \\
    \item Note that the empty tuple $U = (\ \ , \ \ , \ \ , \ \ )$ is a specific form of unmarked state for tuples. We can interpret the usual Laws of Form mark as acting on this void without ambiguity.\\
    \item This consideration does not quite apply to every  type of mark enclosing the empty state. For example, the expression $\cross{\ \ }_i$  is operating on the empty state, and we could write $\cross{\ \ }_i = \cross{(\ \ , \ \ , \ \ , \ \ )}_i = (\cross{\ }, \ \ , \ \ , \cross{\ }).$  However, the left side of this equation is an operator, while the right side apparently is not. In order to permit this equivalence, we resolve  the discrepancy by adopting the following definitions  incorporating exponential forms of operation:  
         \begin{equation}
            \cross{\ \ }_i = (\cross{\  },\ ,\ ,\cross{\ })
        \end{equation}
        \begin{equation}
            (a,b,c,d)^{\cross{\ }_i}= \cross{(a,b,c,d)}_i = (\cross{b},a ,d , \cross{c})
        \end{equation}
        \begin{equation}
            (a,b,c,d)^{ (\cross{\  },\ ,\ ,\cross{\ })}= (a,b,c,d)^{\cross{\ }_i} = (\cross{b},a ,d , \cross{c}).
        \end{equation}
\end{itemize}

A second consideration involves exponents.   
\begin{itemize}
    \item For $\cross{\cross{X}_i}_i = \cross{X}$, we can equivalently write $\cross{X}_i^2 = \cross{X}$, and likewise the same applies for $\cross{X}_j^2 = \cross{X}$ and $\cross{X}_k^2 = \cross{X}$.\\
    \item Furthermore, $\cross{\cross{\cross{\cross{X}_i}_i}_i}_i$ can be re-written as $\cross{X}_i^4$.  Since $\cross{\cross{X}_i}_i = \cross{X}$, it follows immediately that $\cross{X}_i^4 = \cross{\cross{X}} = X$, which is true for $\cross{X}_j^4$ and $\cross{X}_k^4$ as well.\\
    \item Finally, we see that $\cross{X}_i^3 = \cross{\cross{\cross{X}_i}_i}_i = \cross{\cross{X}_i}$, so we can write $\cross{X}_i^3$ for $\cross{\cross{X}_i}$ and $\cross{\cross{X}}_i$.  The same applies for $\cross{\ \ }_j^3$ and $\cross{\ \ }_k^3$.
\end{itemize}

\section{The BF Calculus}\label{sec30}
The Q Calculus is closely related to the BF Calculus, in the same way that the quaternions are related to the complex numbers. $BF$ is a 4-valued extension to $LoF$, defined as follows. 
\begin{definition}
\begin{enumerate}
    \item Expressions in $BF$ take the form of pairs $(a,b)$, where $a$ and $b$ are expressions in $LoF$ .\\
    \item The mark \cross{(a,b)} is re-defined, with  
    \begin{equation}
        \cross{(a,b)} = (\cross{a}, \cross{b}).
    \end{equation}
\item  The juxtaposition of two expressions $A$ and $B$, is defined on  a pairwise basis, so that 
    \begin{equation}
        AB = (a_1, a_2)(b_1, b_2) = (a_1\ b_1,\ a_2\ b_2).\\[6pt]
    \end{equation}
    \item The $imaginary$ mark $\cross{\ \ }_i$, or \emph{Square Root of Negation}, is introduced as a new enclosure-based operator, with the following definition 
    \begin{equation}
        \cross{(a,b)}_i = (\cross{b}, a),
    \end{equation}
    which satisfies the equation 
        \begin{equation}
            \cross{(a,b)}_i^2 = \cross{(a,b)}.
    \end{equation}
\end{enumerate}
\end{definition}
The imaginary mark corresponds directly to Hamilton\textquotesingle s representation of multiplying imaginary numbers, which takes the form of  rotations in the complex plane: 
\begin{equation}
    (a,b)i = (-b,a).
\end{equation}
The term \emph{Square Root of Negation} was originally suggested by Kauffman\cite{kau4} and later developed by the authors as $BF$\cite{BF3, BF4}. Important results described in these papers include proving completeness theorems for $BF$, establishing that $BF$ is closely related to Kauffman\textquotesingle s \emph{Waveform Algebra}\cite{kau1} and Belnap\textquotesingle s bilattice \emph{Four}\cite{BILx1, BILx4, BILx2}, while further establishing that certain normal modal operations\cite{Mod1, Mod5} that can be expressed in the form of a bilattice can also be constructed as expressions in BF\cite{bfmodal}. A key aspect of $BF$ is that it satisfies the form of \emph{Split Generation}, a form that does not exist in LoF. 
\begin{equation}
    Split\ Generation. \quad \quad \cross{\cross{A}_i\ B}_i\ C = \cross{\cross{AC}_i\ B}_i\ C
\end{equation}

\subsection{The I, J, and K subspaces of Q}\label{sec41}
Consider the set $I$ of expressions in $Q$ that include only  $\cross{X}_i$ and $\cross{X}$ as containment operators, with  no instances of $\cross{X}_j$ or $\cross{X}_k$. Then $I$ can be regarded as a 16-valued extension of $BF$, and it follows that if $F$ and $G$ are two expressions such that $F = G$ in $BF$, then $F = G$ in $I$. The same is also true for the corresponding sets $J$ (expressions with no instances of  $\cross{X}_i$ or $\cross{X}_k$) and $K$ (expressions with no instances of $\cross{X}_i$ or $\cross{X}_j$).\\

In this regard, the $Q$ Calculus can be regarded as three distinct but interrelating sub-algebras, each of which is a 16-valued instance of $BF$, but none of which are complete in and of themselves. Lack of completeness of the individual I, J, and K sub-algebras can be easily discerned by  observing that the permutations underlying BF are strictly limited to the single type $(a,b) \longrightarrow (b,a)$, whereas the  permutations underlying Q  include  the entire Klein Four group, as will be  discussed further in Section \ref{sec50}. In the appendix, we show by example that arbitrary permutations can be constructed in the form of Q expressions.  

\section{Distribution Axioms}\label{sec40}
The number and character of distribution laws in $Q$ is quite remarkable, and this gives us incentive to examine these laws carefully, highlighting the role played by non-commutative marks in their demonstration. 
\subsection{LoF}
We observe that in LoF the following distribution equations are satisfied, based on LoF\textquotesingle s usual interpretation that $A \vee B = AB$ and $A \wedge B = \cross{\cross{A}\cross{B}} $.

\begin{equation*}
    (a \wedge b) \vee c = (a \vee c) \wedge (b\vee c)
\end{equation*}
\begin{equation*}
    (a \vee b) \wedge c = (a \wedge c) \vee (b \wedge c) . 
\end{equation*}\\
The first corresponds to Spencer-Brown\textquotesingle s form of Transposition  (which he assumes as an algebraic initial), while the second  is less frequently seen but is easily demonstrated. 

\begin{align}
    \ \quad \quad \cross{\cross{A}\cross{B}}C &= \cross{\cross{AC}\cross{BC}}\\
    \quad \quad \cross{\cross{AB}\cross{C}} &= \cross{\cross{A}\cross{C}}\ \cross{\cross{B}\cross{C}}.
    \label{eqdist2}
\end{align}

Two additional distribution instances, 
   $$(a \wedge b) \wedge c = (a \wedge c ) \wedge (b \wedge c)$$
    $$(a \vee b) \vee c = (a \vee c) \vee (b \vee c)$$

\noindent
are also satisfied, but they are trivial. Ignoring the trivial cases, we conclude that LoF satisfies two distribution laws.

 \subsection{BF}

Based on its underlying bilattice characteristics, $BF$ adds two new logical operations, which we will designate as $\vee_i$ and $\wedge_i$. 

\begin{equation}
    A \vee_i B = \cross{ \cross{\ A}_i^3 \ \cross{\ B}_i^3 \ }_i
    \label{veei}
\end{equation} 
\begin{equation}
    A \wedge_i B = \cross{\cross{\ A}_i \cross{\ B}_i }_i^3
    \label{wedgei}
\end{equation}

These additions increase the number of conjunction and disjunction operators to four: $\wedge,\ \vee,\ \wedge_i,\ \vee_i$. Ignoring the four trivial cases, the number of distribution laws increases accordingly to $n \times (n-1) = 4 \times 3 = 12$.

\begin{equation}
    \cross{ \cross{\ A}_i^3 \ \cross{\ B}_i^3 \ }_i \ C = \cross{ \cross{\ A\ C}_i^3 \ \cross{\ B\ C}_i^3 \ }_i 
    \label{eqdistbf1}
\end{equation} 
\begin{equation}
    \cross{\cross{\ A}_i \cross{\ B}_i }_i^3\ C = \cross{\cross{\ A\ C}_i \cross{\ B\ C}_i }_i^3 
    \label{eqdistbf2}
\end{equation}\\
Equation (\ref{eqdistbf1}) corresponds with $(a \vee_i b)\ \vee c$, while Equation (\ref{eqdistbf2}) corresponds with $(a \wedge_i b)\ \vee c$. All 12  distributive combinations are demonstrable in $BF$. 

It helps to understand that $\vee_i$ and $\wedge_i$ operate like disjunction and conjunction in the bilattice context, which is to say, when it is assumed that $True = (\ , \cross{\ \ })$ and $False = (\cross{\ \ }, \ )$. We strongly recommend that the reader review our discussions of bilattices in (\citenum{BF3}) and (\citenum{BF4}). 

\subsection{Q}

$Q$ adds four additional logical operations, which we will designate as 
\begin{equation}
    A \vee_j B = \cross{ \cross{\ A}_j^3 \ \cross{\ B}_j^3 \ }_j 
    \label{veej}
\end{equation} 
\begin{equation}
    A \wedge_j B = \cross{\cross{\ A}_j \cross{\ B}_j }_j^3
    \label{wedgej}
\end{equation}
\begin{equation}
    A \vee_k B = \cross{ \cross{\ A}_k^3 \ \cross{\ B}_k^3 \ }_k 
    \label{veek}
\end{equation} 
\begin{equation}
    A \wedge_k B = \cross{\cross{\ A}_k \cross{\ B}_k }_k^3
    \label{wedgek}
\end{equation}

Equation (\ref{eqdistbf1}) and (\ref{eqdistbf2}) are also valid for the $\cross{\ \ }_j$ and $\cross{\ \ }_k$ operators. Consequently,  

\begin{equation}
    \cross{ \cross{\ A}_\alpha^3 \ \cross{\ B}_\alpha^3 \ }_\alpha \ C = \cross{ \cross{\ A\ C}_\alpha^3 \ \cross{\ B\ C}_\alpha^3 \ }_\alpha 
    \label{eqdistbf3}
\end{equation}

\begin{equation}
    \cross{\cross{\ A}_\alpha \cross{\ B}_\alpha }_\alpha^3\ C = \cross{\cross{\ A\ C}_\alpha \cross{\ B\ C}_\alpha }_\alpha^3 
    \label{eqdistbf4}
\end{equation}\\
for $\alpha \in \{i,j,k\}$.
ignoring the trivial cases, the number of distribution laws increases in $Q$ to $56\ (8\times 7).$ The following demonstrations are particularly interesting due to their use of non-commuting quaternion operations. \\ 

$A \vee_{i} (B \wedge_{j} C)$\\[10pt]
\begin{tabular}{@{\hspace{7ex}}l l }
$ = \cross{\cross{A}^3_i\ \cross{ B \wedge_j C}_i^3\ }_i$ &[Translation per (\ref{veei})] \\[5pt]
$ = \cross{\cross{A}^3_i\ \cross{\ \cross{\cross{ B}_j\ \cross{ C}_j  }^3_j\ }_i^3\ }_i$ &[Translation per (\ref{wedgei})] \\[5pt]
$ = \cross{\cross{A}_i^3\  \cross{\ \cross{B}_j\ \cross{ C}_j \ }^3_k\ }_i$ &[$-j \times -i = -k$] \\[5pt]
$ = \cross{\cross{A}_i^3\  \cross{\ \cross{ \cross{B}_i^3}_k\ \cross{ \cross{ C}_i^3}_k \ }^3_k\ }_i$ &[$j = -i \times k$] \\[5pt]
$ = \cross{  \cross{\ \cross{ \ \cross{A}_i^3\ \cross{B}_i^3}_k\ \cross{\ \cross{A}^3_i\ \cross{ C}_i^3}_k \ }^3_k\ }_i$ &[Applying Eq. (\ref{eqdistbf4}), $\alpha = k$] \\[5pt]
$ =  \cross{  \cross{ \ \cross{A}_i^3\ \cross{B}_i^3}_k\ \cross{\ \cross{A}_i^3\ \cross{ C}_i^3}_k }^3_j$ &[$-k\times i = - j$] \\[5pt]
$ =  \cross{  \cross{ \ \cross{ \cross{A}_i^3\ \cross{B}_i^3\ }_i}_j\  \cross{\ \cross{\ \cross{A}_i^3\ \cross{ C}_i^3\ }_i }_j }^3_j$ &[$k = i \times j$] \\[5pt]
$= (A \vee_i B ) \wedge_j (A \vee_i C)$ &[Translation]\\[10pt]
\end{tabular}\\[5pt]

\newpage

$(A \wedge_k B) \wedge_j C$\\[10pt]
\begin{tabular}{@{\hspace{7ex}}l l }
    $= \cross{\cross{A \wedge_k B}_j\ \cross{C}_j}_j^3$ &[Translation per (\ref{wedgej})] \\[5pt]
$= \cross{\cross{\cross{\cross{A}_k \ \cross{B}_k\ }_k^3}_j\ \cross{C}_j}_j^3$ &[Translation per (\ref{wedgek})] \\[5pt]
$= \cross{\cross{\cross{A}_k \ \cross{B}_k\ }_i\ \cross{C}_j}_j^3$ &[$-k \times j = i$] \\[5pt]
$= \cross{\cross{\cross{\cross{A}_j}_i^3 \ \cross{\cross{B}_j}_i^3\ }_i\ \cross{C}_j}_j^3$ &[$k = j \times -i $] \\[5pt]
$= \cross{\cross{\cross{\cross{A}_j\ \cross{C}_j}_i^3 \ \cross{\cross{B}_j\ \cross{C}_j}_i^3\ }_i\ }_j^3$ &[Applying Eq. (\ref{eqdistbf3}), $\alpha = i$] \\[5pt]
$= \cross{\cross{\cross{A}_j\ \cross{C}_j}_i^3 \ \cross{\cross{B}_j\ \cross{C}_j}_i^3\ }_k^3$ &[$i\times -j= - k$] \\[5pt]
$= \cross{\cross{\cross{\cross{A}_j\ \cross{C}_j}_j^3}_k \ \cross{\cross{\cross{B}_j\ \cross{C}_j}_j^3}_k\ }_k^3$ &[$-i = -j \times k$] \\[5pt]
$= (A \wedge_j C ) \wedge_k (B \wedge_j C)$ &[Translation]\\[10pt]
\end{tabular}\\[5pt]

We note the Q Calculus is technically isomorphic to the 16-valued lattice described by Shrampko and Wansing.\cite{BILx6} However the character and instantiation of these two calculi are vastly different. The particular power of the Q formulation derives from its connection to the quaternions and from the fact that every expression in Q resolves directly  as LoF equations within the n-tuples.

\section{Permutations}\label{sec50}
As suggested in Section \ref{sec41}, permutations play a significant role in $Q$. In this section we examine the permutations that generate the Klein Four Group, which is closely associated with the structure of the quaternions. 

In Figure \ref{fig:ff1} we see a diagrammatic representation of the Klein Four Group. The group action occurs via the composing of permutations. As an example, in Figure \ref{fig:AAE} we see the composition of A with itself, which is equal to the identity $E$. It is not shown in the figures, but $BB = E$ and $CC = E$, so each permutation is its own inverse. 

 \begin{figure}[htb]
    \centering
    \includegraphics[width=2.5in]{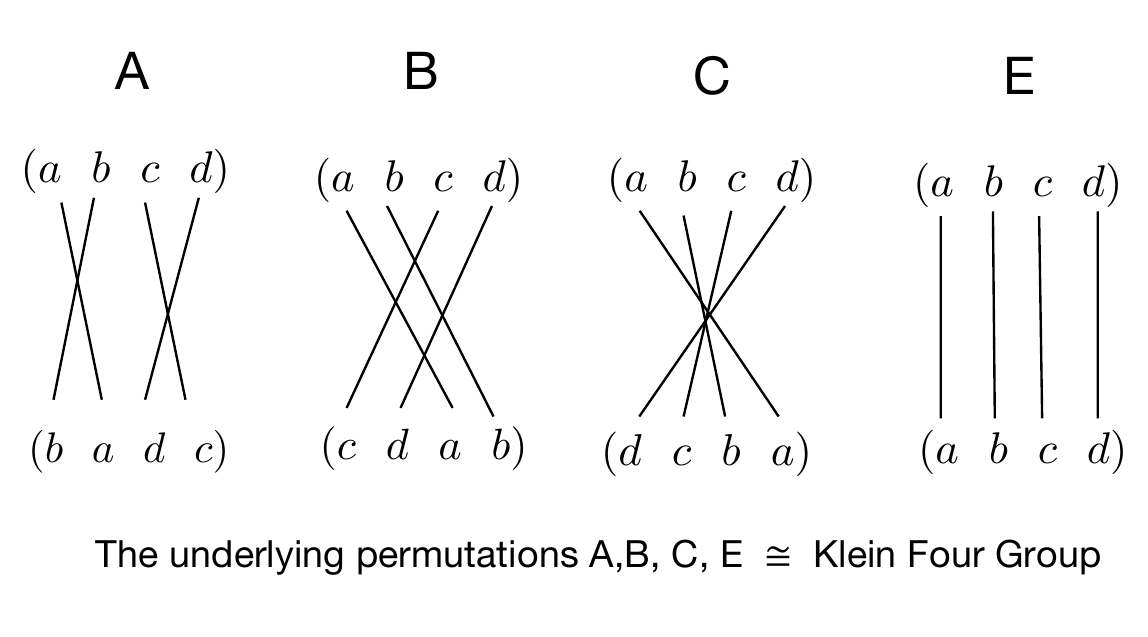}
    \caption{The permutations A, B, C form a group with identity E, that is isomorphic to the Klein Four Group.}
    \label{fig:ff1}
\end{figure}
 \begin{figure}[htb]
    \centering
    \includegraphics[width=1.75in]{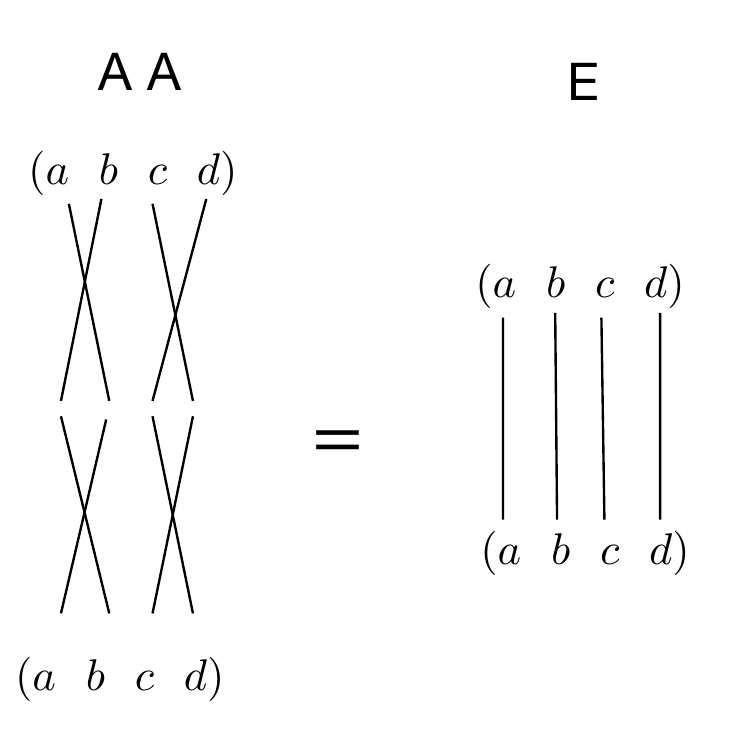}
    \caption{AA = E: Permutation A composed with itself results in the identity E.}
    \label{fig:AAE}
\end{figure}

We also find that $AB = C$ (as shown in Figure \ref{fig:ABC}), $AC = B$, and $BC = A$, and also that $AB = BA$, $AC = CA$, and $BC = CB$. 
 \begin{figure}[H]
    \centering
    \includegraphics[width=1.75in]{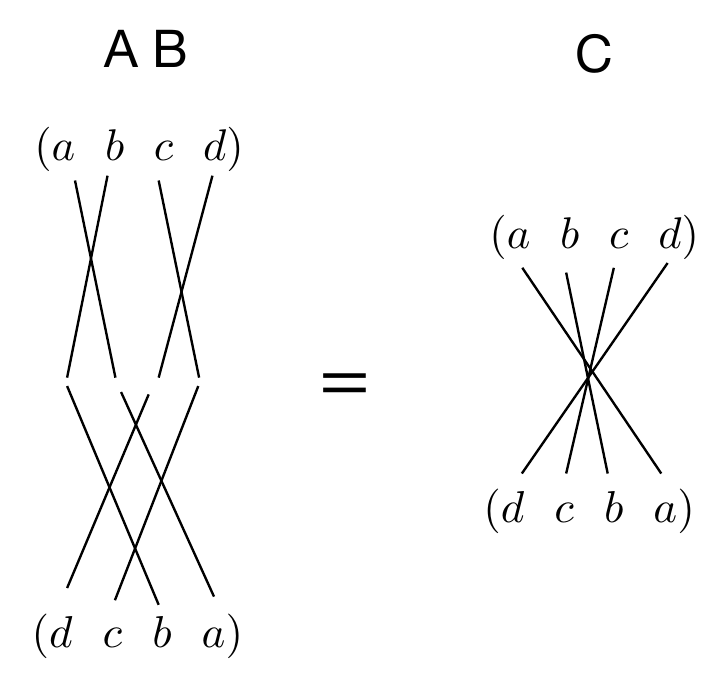}
    \caption{The composition AB = C. It's also the case that AC = B and BC = A. }
    \label{fig:ABC}
\end{figure}
 \begin{figure}[H]
    \centering
    \includegraphics[width=2.5in]{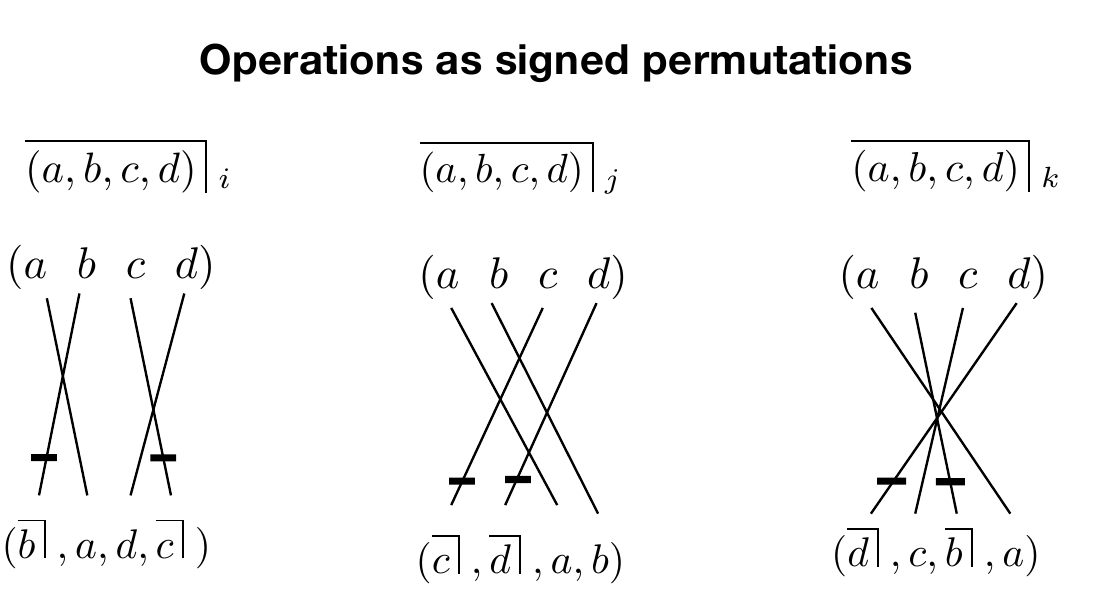}
    \caption{The $Q$ operators are expressed as signed permutations. The algebraic and diagrammatic forms match closely.}
    \label{fig:ff2}
\end{figure}
 In the Figure \ref{fig:ff2} we show the $Q$ quaternion operations $\cross{\ \ }_i$, $\cross{\ \ }_j$, and $\cross{\ \ }_k$ as signed permutations. The reader can observe that the permutations underlying the quaternion operations directly correspond to $A,\ B, \ C$ in Figure \ref{fig:ff1} and that a dark bar on a permutation line corresponds to the Spencer-Brown cross so that a signal going through a bar is operated on by the mark.
  \begin{figure}[htb]
    \centering
    \includegraphics[width=1.75in]{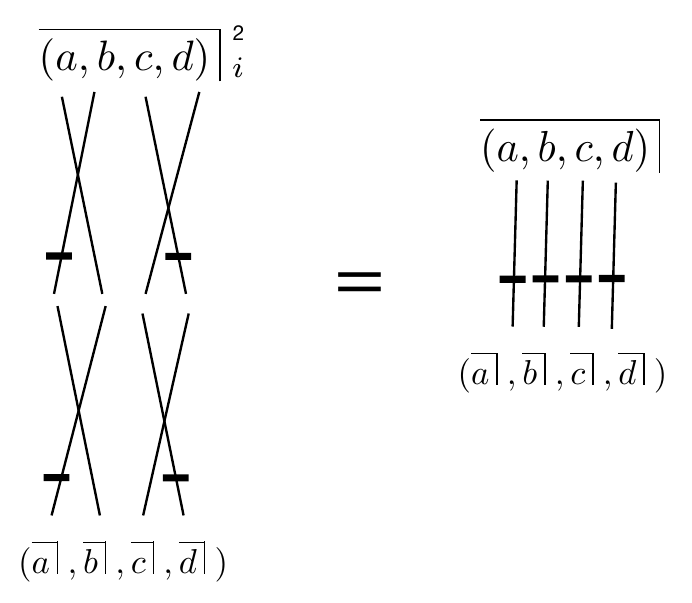}
    \caption{The composition $\cross{(a,b,c,d) }_i^2 = \cross{(a,b,c,d)}. $}
    \label{fig:IIM}
\end{figure}
  \begin{figure}[htb]
    \centering
    \includegraphics[width=2.25in]{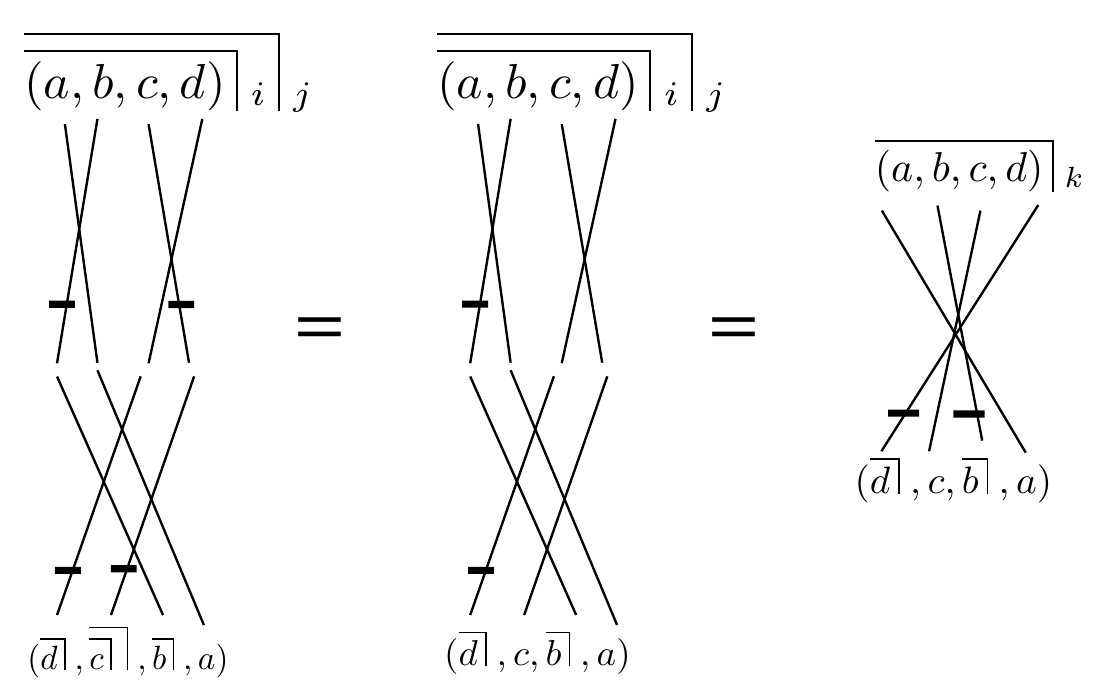}
    \caption{Illustrating that $ij = k$. Note the c-strand has two signs, which cancel in accord with $\cross{\cross{c}} = c$, within LoF.}
    \label{fig:IJJI}
\end{figure}

The diagram in Figure \ref{fig:IIM} shows the composition of $\cross{\cross{\ \ }_i}_i = \cross{\ \ }$ , while  Figure \ref{fig:IJJI} shows the composition $\cross{\cross{\ \ }_i}_j = \cross{\ \ }_k$.

 \begin{figure}[H]
    \centering
    \includegraphics[width=2.5in]{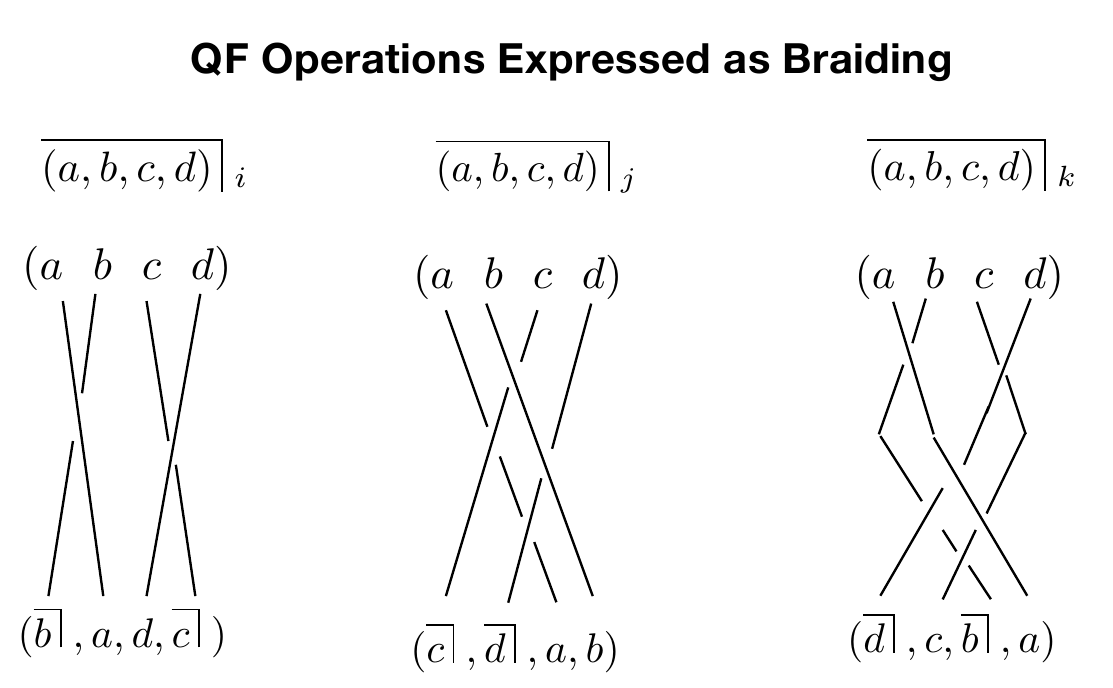}
    \caption{In the diagram, the Q quaternion operators are expressed as braid crossings.}
    \label{fig:ff3}
\end{figure}

 \begin{figure}[H]
    \centering
    \includegraphics[width=3.75in]{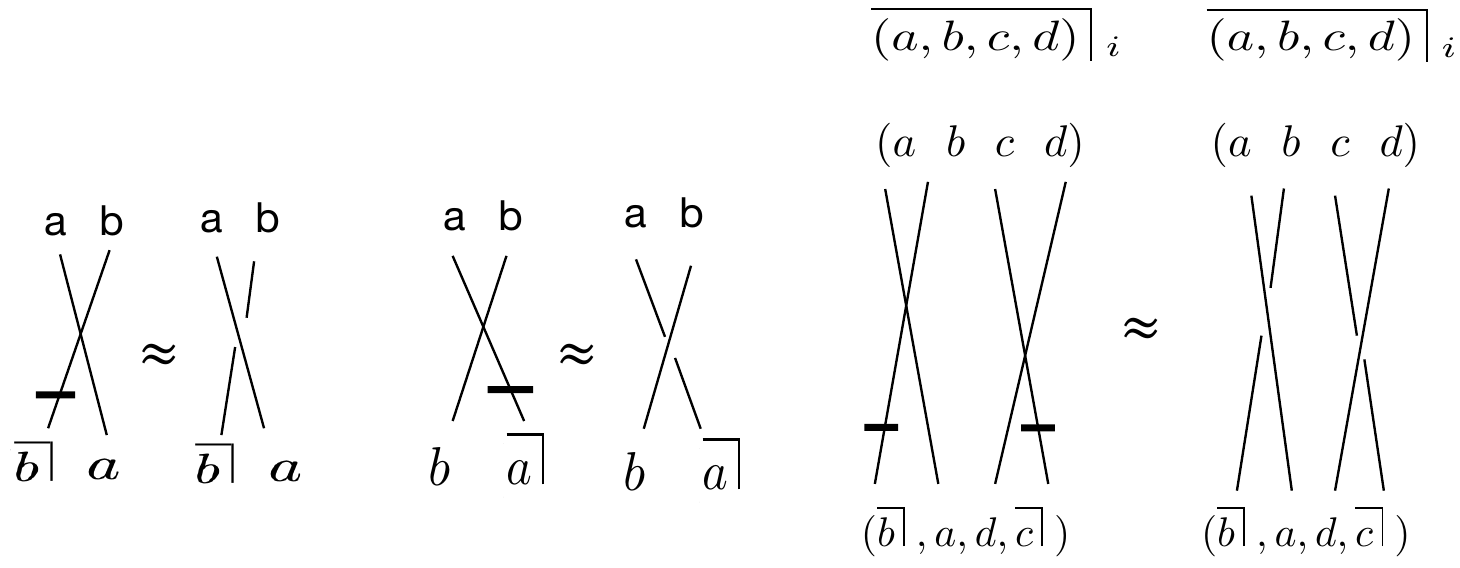}
    \caption{Signed Permutation vs. Braided Representations: The bar indicates a marked tuple, which corresponds to an under-crossing in the braided representation.}
    \label{fig:ff4}
\end{figure}

In Figure \ref{fig:ff3}, the operations 
$\cross{\ \ }_i$, $\cross{\ \ }_j$, and $\cross{\ \ }_k$ are expressed as braiding operations. Figure \ref{fig:ff4} highlights the correspondence between marked strands and under-crossings. The three quaternion operations in Figure \ref{fig:ff3} can be represented as braid moves:

\begin{itemize}
    \item $\cross{(a,b,c,d)}_i \  \longleftrightarrow \  \sigma_1 \sigma_3^{-1}$\\
    \item $\cross{(a,b,c,d)}_j \ \longleftrightarrow \  \sigma_2 \sigma_1^{-1} \sigma_3 \sigma_2^{-1}$\\
   \item $\cross{(a,b,c,d)}_k \ \longleftrightarrow \  \sigma_1 \sigma_3^{-1} \sigma_2 \sigma_1^{-1} \sigma_3 \sigma_2^{-1}$
\end{itemize}
where $\sigma_1$ indicates that the 1st strand crosses $over$ the 2nd , while $\sigma_3^{-1}$ indicates that the 3rd strand crosses $under$ the 4th. The braid words above for the quaternion operators are exactly a transcription of the braids drawn in Figure \ref{fig:ff3}.

We suggest that the reader carefully follow the braid moves while referring to the  diagrams in Figure \ref{fig:ff3}. Once a degree of familiarity is obtained one can use the braiding moves as a way to multiply the quaternion operations, for example by confirming the composition $ij = k$.   \\

The Artin Braid  group on n strands\cite{Braid1, Braid2} is defined algebraically by  the generators $\sigma_1, \sigma_2, \dots \sigma_{n-1}$ and the relations\\

\noindent
\emph{Artin Braid Group - Relations}
\begin{enumerate}
    \item $\sigma_i\sigma_j = \sigma_j\sigma_i \ for\ |i - j| > 1$, 
    \item $\sigma_i\sigma_j\sigma_i = \sigma_j\sigma_i\sigma_j\ for\ |i - j| = 1$.\\
\end{enumerate}
   Figure \ref{fig:sigman1}  illustrates braid relation 2 in the form  $\sigma_1^{-1} \sigma_2^{-1} \sigma_1^{-1} = \sigma_2^{-1} \sigma_1^{-1} \sigma_2^{-1}. $ 
  \begin{figure}[htb]
    \centering
    \includegraphics[width=2.25 in]{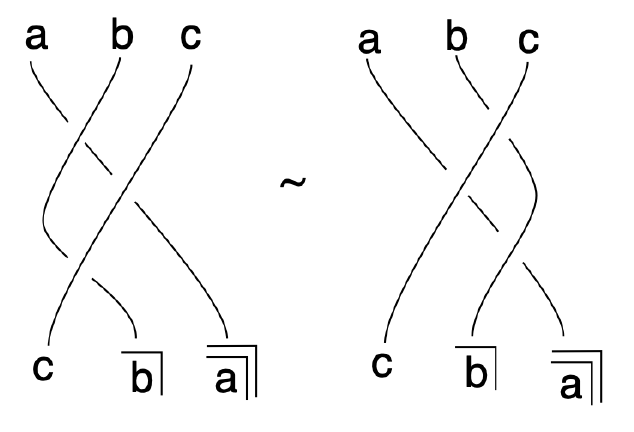}
    \caption{The braid relation $\sigma_1 \sigma_2 \sigma_1 = \sigma_2 \sigma_1 \sigma_2$ }
    \label{fig:sigman1}
\end{figure}\\ 

We will define a representation of $B_n$ as operations on LoF n-tuples by 

\begin{equation}
    (a_1,a2, \dots,a_{n-1},a_n)^{\sigma_k} = (a_1, \dots a_{k-1},\cross{a_{k+1}}, a_k, a_{k+1}, \dots a_{n-1}, a_n) 
\end{equation}\\

In contrast to the Artin Braid group,  the LoF representation incorporates the square root of negation and is finite. See Figure \ref{fig:sigman4}. \\

\noindent
\emph{ LoF Braid Group Representation}
\begin{enumerate}
    \item $   \cross{A}\sigma_k = (a_1,\ a_2,\ \dots \ , \cross{a_{k+1}},\ a_{k},\ \dots\  ,a_n),$
    \item $   \cross{A}\sigma_k^2 = (a_1,\ a_2,\ \dots \ , \cross{a_k},\ \cross{a_{k+1}},\ \dots\  ,a_n),$
    \item $\sigma_k^4 = 1$,
    \item $\sigma_i\sigma_j = \sigma_j\sigma_i \ for\ |i - j| > 1$, 
    \item $\sigma_i\sigma_j\sigma_i = \sigma_j\sigma_i\sigma_j\ for\ |i - j| = 1$.\\
\end{enumerate}

The LoF braid group representation is the context for our 4-strand quaternion representation. The  reader can verify that this operation satisfies the standard braiding relations and that $\sigma_k^2$ acts to cross just the elements in the tuple at $k$ and {k+1} so that $\sigma_i^4 = 1$. The reader may also enjoy confirming the basic quaternion operations by using braiding relations to prove them as transformations. The authors find it illuminating that the BF operator is a  braid.

As shown in Figure \ref{fig:sigman2}, the braid notation forms the basis to generalize Q to  orders of $n$ braid strands. In this figure, strands $k$ and $k+1$ are subject to the operation $\sigma_k$, which we are representing in LoF\textquotesingle s contaiment notation by $\cross{A}\sigma_k$. Note that the $k$ and $k+1$ braids are the only ones being operated on, as is made very clear in the corresponding n-tuple. 

  \begin{figure}[htb]
    \centering
    \includegraphics[width=3.0 in]{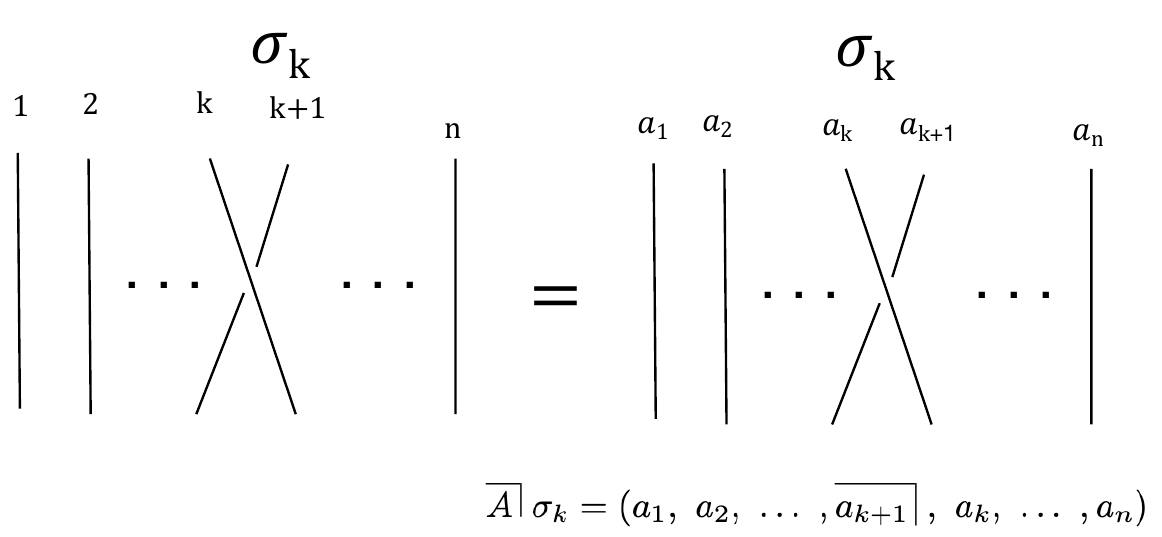}
    \caption{LoF braid representation, generalized to n strands. }
    \label{fig:sigman2}
\end{figure}

Also note that that the notation for $\sigma_k$ and $\sigma_k^{-1}$ are closely related, as can be seen in Equations \ref{eq34} and \ref{eq35}.

\begin{equation}\label{eq34}
    \cross{A}\sigma_k = (a_1,\ a_2,\ \dots \ , \cross{a_{k+1}},\ a_k,\ \dots\  ,a_n)
\end{equation}
\begin{equation}\label{eq35}
    \cross{A}\sigma_k^{-1} = (a_1,\ a_2,\ \dots \ , \ a_{k+1},\ \cross{a_k},\ \dots\  ,a_n)
\end{equation}
If instead $\cross{A}\sigma^{-1}$ is called for, as it is in Equation \ref{eq35}, then the diagram is adjusted so that braid $k$ crosses under $k+1$, as shown in Figure \ref{fig:sigman3}. 

  \begin{figure}[htb]
    \centering
    \includegraphics[width=3.0 in]{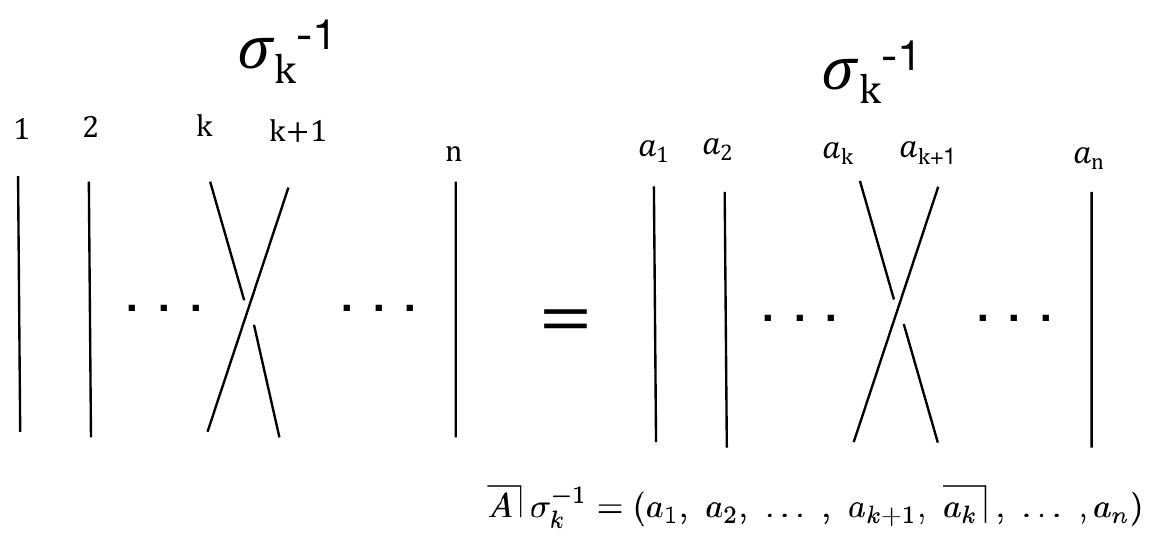}
    \caption{Positions of braid k (undercrossing) and braid k+1 (overcrossing) are reversed from Figure \ref{fig:sigman2}, since the operation $\cross{A}\sigma_k^{-1}$ rather than $\cross{A}\sigma_k$ is now represented. }
    \label{fig:sigman3}
\end{figure}

Figure \ref{fig:sigman4} illustrates the incorporation of the square root of the mark in our representation of the braid group. We encourage the reader to review the diagrams for themselves, and in particular to confirm that the braid operations shown in Figure \ref{fig:ff3} correspond to the quaternion operations as we have defined them. \\

  \begin{figure}[htb]
    \centering
    \includegraphics[width=2.75 in]{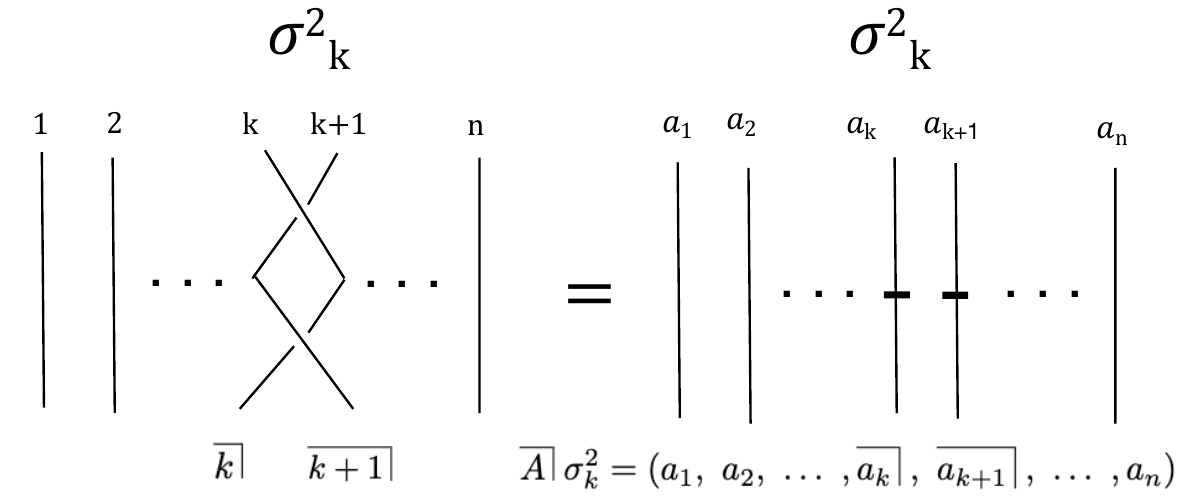}
    \caption{Applying the $\sigma_k$ operation twice ($\cross{A}\sigma_k^2$) results in marking both the k and k+1 braid strands. Applying $\sigma_k$ four times marks the k and k+1 twice each, which cancel, giving $\cross{A}\sigma_k^4 = 1$. }
    \label{fig:sigman4}
\end{figure}

An extremely pragmatic aspect of the Q calculus is the power it affords to access the LoF algebra within the n-tuples on an underlying basis.  By representing braids as operations on n-tuples we have obtained an LoF instantiation of the n-strand braid group that is analog to our quaternion representation. Diagrams \ref{fig:ff1}-\ref{fig:sigman4} are meant to give the reader a deeper, intuitive sense of the interoperability of the quaternion and braiding operations. There is much more to say about the connection  between \emph{Laws of
Form}, braiding, and the quaternions\cite{kau2, kau5}. Of particular interest are establishing  algebraic completeness for Q, extending its representation of modal operators, and exploring the degree to which the concepts of nilpotency\cite{zinfinity} and a calculus of fermions\cite{kau3, kau9} are  embedded in its logical structure.  We anticipate taking up these ideas in a longer paper focused  on algebraics and  extensions to Clifford Algebras\cite{Braid3} and other hypercomplex systems.  \\

\newpage

\appendix 

\noindent
$\mathbf{Appendix}$\\

The first two sections of this appendix consist of lists of initials and consequences for LoF and Q. The astute reader will observe that we have not chosen to distinguish initials (axioms) from consequences.  What is important to emphasize is that all initials and consequences for LoF, including the 10 enumerated below, are valid in the Q Calculus, based on the understanding that variables $A,B,C$ each represent four-tuples, e.g. $A = (a_1, a_2, a_3, a_4)$, $B = (b_1, b_2, b_3, b_4)$, and $C = (c_1, c_2, c_3, c_4)$. \\

The second section consists of initials and consequences that are valid in Q but not (in general) in LoF. \\

Finally, in the third section we illustrate via examples that expressions can be constructed subjecting an arbitrary tuple to a mark, as well as subjecting a given 4-tuple to an arbitrary permutation.

\section{Initials and Consequences -- Valid in both LoF and Q} \label{App:AppendixA}

\begin{tabular}{@{\hspace{0ex}}l l }
$\mathbf{1.}\ \  \cross{\cross{A}\ A } = \quad \quad .\ $ &(Position) \\[5pt]
$\mathbf{2.}\ \ \cross{\cross{A} \cross{B}}C = \cross{\cross{AC} \cross{BC}}$ &(Transposition) \\[5pt]
$\mathbf{3.} \ \ \cross{\cross{A}} =  A$ &(Reflexion) \\[5pt]
$\mathbf{4.} \ \  \cross{A}B =  \cross{AB}B$ &(Generation) \\[5pt]
$\mathbf{5.} \ \ A\ \cross{\ \ } =  \cross{\ \ }$ &(Integration) \\[5pt]
$\mathbf{6.} \ \ \cross{\cross{A} B}A =  A$ &(Occultation) \\[5pt]
$\mathbf{7.} \ \ A\ A =  A $ &(Iteration) \\[5pt]
$\mathbf{8.} \ \ \cross{\cross{A}\cross{B}}\  \cross{\cross{A} B} =  A $ &(Extension) \\[5pt]
$\mathbf{9.} \ \ \cross{\cross{\cross{A}B}C} = \cross{AC}\ \cross{\cross{B}C}  $ &(Echelon) \\[5pt]
$\mathbf{10.} \ \ \cross{\cross{\cross{A}B}\  \cross{\cross{A}\cross{B}}} = \cross{AB} \cross{A\cross{B}} $ &(Crosstransposition) \\[5pt]
\end{tabular}\\[10pt]

\newpage
\section{Initials and Consequences -- Valid  in Q Only}
\noindent
(where $\alpha,\ \beta \in \{i,j,k\})$\\

\begin{tabular}{@{\hspace{0ex}}l l }
$\mathbf{Q1.} \ \ \cross{\cross{A}_\alpha }_\alpha = \cross{A}  $ &(SQR \cross{\ \ }) \\[5pt]
$\mathbf{Q2.} \ \ \cross{\cross{\cross{A}_i}_j }_k = \cross{A}  $ &(IJK) \\[5pt]
$\mathbf{Q3.} \ \ \cross{A}_\alpha^4  = A  $ &(Quadra Reflexion) \\[5pt]
$\mathbf{Q4.} \ \ \cross{\cross{A}_\alpha}  = \cross{\cross{A}}_\alpha   $ &(\cross{\ \ } Commutes) \\[5pt]
$\mathbf{Q5.} \ \ \cross{\cross{A}_\alpha}_\beta  = \cross{\cross{\cross{A}_\beta}_\alpha }  $ &$(\cross{\ \ }_{\alpha \beta}$ Anti-commutes) \\[5pt]
$\mathbf{Q6.} \ \ \cross{\cross{A}_\alpha \ B}_\alpha\ C  =  \ \ \cross{\cross{AC}_\alpha \ B}_\alpha\ C  $ &(Split Generation) \\[5pt]
$\mathbf{Q7.} \ \ \cross{A\ \cross{\ \ }_\alpha}_\alpha    = \cross{A}_\alpha \ \cross{\ \ }_\alpha^3 $ &(Extraction) \\[5pt]
$\mathbf{Q8.} \ \ \cross{AB}_\alpha    = \cross{\cross{\cross{A}_\alpha \cross{B}_\alpha} \ \cross{\cross{A}_\alpha \ \cross{\ \ }_\alpha^3}\ \cross{\cross{B}_\alpha \cross{\ \ }_\alpha^3 } }$ &(Disintegration) \\[5pt]
$\mathbf{Q9.} \ \ \cross{ \cross{\ A}_\alpha^3 \ \cross{\ B}_\alpha^3 \ }_\alpha \ C = \cross{ \cross{\ A\ C}_\alpha^3 \ \cross{\ B\ C}_\alpha^3 \ }_\alpha $ &(Right Distribution) \\[5pt]
$\mathbf{Q10.} \ \ \cross{ \cross{\ A}_\alpha^3 \ \cross{\ B}_\alpha^3 \ }_\alpha \ C = \cross{ \cross{\ A\ C}_\alpha^3 \ \cross{\ B\ C}_\alpha^3 \ }_\alpha $ &(Left Distribution) \\[5pt]
$\mathbf{Q11.} \ \ \cross{ \cross{\ \ }_i \cross{\ \ }_j} \cross{ \cross{\ \ }_i^3\ \cross{\ \ }_j^3} = \cross{\ \ }_k $ &(Compile-k) \\[5pt]
$\mathbf{Q12.} \ \ \cross{ \cross{\ \ }_j \cross{\ \ }_k} \cross{ \cross{\ \ }_j^3\ \cross{\ \ }_k^3} = \cross{\ \ }_i $ &(Compile-i) \\[5pt]
$\mathbf{Q13.} \ \ \cross{ \cross{\ \ }_i \cross{\ \ }_k} \cross{ \cross{\ \ }_i^3\ \cross{\ \ }_k^3} = \cross{\ \ }_j $ &(Compile-j) \\[5pt]
\end{tabular}\\[5pt]
Note that equations Q11-Q13 are strictly arithmetic. They apply only to equations with  empty marks $\cross{\ \ }_i,\ \cross{\ \ }_j$ and $\cross{\ \ }_k$, and not more generally to equations that include $\cross{X}_i,\ \cross{X}_j$, or $\cross{X}_k$

\section{Example Constructions in Q}

The two demonstrations in this section illustrate the existence of the following types of expressions in Q: 
\begin{enumerate}
    \item Expressions that enclose an arbitrary single tuple within a mark. 
    \item Expressions that subject tuples (a,b,c,d) to an arbitrary permutation.  
\end{enumerate}
In these expressions, the marking pattern    interferes with the input values of the expression. For example, the combination $\cross{\ \ }_i\ \cross{\ \ }_j = (\cross{\ }, \ , \ , \cross{\ })\ (\cross{\ }, \cross{\ }, \ , \ ) = (\cross{\ }, \cross{\ }, \ , \cross{\ } )$ creates an interference pattern that blocks $a,b,d$ but not $c$. For simplicity of reference in the demonstration, we refer to these patterns as: $IJ = \cross{\ \ }_i\ \cross{\ \ }_j$, $IK = \cross{\ \ }_i\ \cross{\ \ }_k$, $ JK = \cross{\ \ }_j\ \cross{\ \ }_k$, and $ I^3J^3 = \cross{\ \ }_i^3\ \cross{\ \ }_j^3$. As discussed in earlier in the paper, these forms of juxtaposition are only given meaning with the formal introduction of 4-tuple operators as defined in Section \ref{sec333}. \\

\noindent
$Example\ 1:\ \cross{X\ \cross{\ \ }_i\ \cross{\ \ }_j}\ \cross{\cross{X} \cross{\cross{\ \ }_i\ \cross{\ \ }_j} } $\\[5pt]
\begin{tabular}{@{\hspace{0ex}}l l }
$= \cross{(a,b,c,d)\  (\cross{\  }, \  ,\ , \cross{\  })\  (\cross{\  }, \cross{\ }  ,\ , \ )}\ \ \cross{\cross{(X)} \cross{\cross{\ \ }_i\ \cross{\ \ }_j} } $ &(IJ) \\[5pt]
$= \cross{(a,b,c,d)\   (\cross{\  }, \cross{\ }  ,\ , \cross{\  }) }\ \ \cross{\cross{(a,b,c,d)} \cross{\cross{\ \ }_i\ \cross{\ \ }_j} } $ &(Def. \ref{def2}, Calling) \\[5pt]
$= \cross{(\cross{\  }, \cross{\ }  ,c\ , \cross{\ } )}\ \ \cross{\cross{(a,b,c,d)} \cross{\cross{\ \ }_i\ \cross{\ \ }_j} } $ &(Def. \ref{def2}, Int.) \\[5pt]
$= (\ ,\ , \cross{c},\ )\ \ \cross{\cross{(a,b,c,d)}\  \cross{\cross{\ \ }_i\ \cross{\ \ }_j} } $ &(Def. 1(\ref{eqcross})) \\[5pt]
$= (\ ,\ , \cross{c},\ )\ \ \cross{\cross{(a,b,c,d)}\  \cross{ (\cross{\  }, \  ,\ , \cross{\  })\  (\cross{\  }, \cross{\ }  ,\ , \ )} } $ &(IJ) \\[5pt]
$= (\ ,\ , \cross{c},\ )\ \ \cross{\cross{(a,b,c,d)}\  \cross{ (\cross{\  }, \cross{\ }  ,\ , \cross{\  })} } $ &(Def. \ref{def2}, Calling) \\[5pt]
$= (\ ,\ , \cross{c},\ )\ \ \cross{(\cross{a},\cross{b},\cross{c},\cross{d})\  (\ , \ , \cross{\  },\  ) } $ &(Def. 1(\ref{eqcross}), Ref.) \\[5pt]
$= (\ ,\ , \cross{c},\ )\ \ \cross{(\cross{a},\cross{b},\cross{\ },\cross{d})\  } $ &(Def. \ref{def2}, Int.) \\[5pt]
$= (\ ,\ , \cross{c},\ )\ \ \ (a,b,\ ,d) $ &(Def. 1(\ref{eqcross}), Ref.) \\[5pt]
$= (a,b,\cross{c}  ,d)  $ &(Def. \ref{def2}) \\[5pt]
\end{tabular}\\[5pt]

\newpage
\noindent
$Example\ 2:\ \cross{\cross{X}\ \cross{\ \ }_i^3\ \cross{\ \ }_j^3}\ \cross{\cross{X}_j \ \cross{\ \ }_i\ \cross{\ \ }_k} \ \cross{\cross{X}_i \cross{\ \ }_j\ \cross{\ \ }_k }\ \cross{\cross{X}_k \cross{\ \ }_i\ \cross{\ \ }_j } $\\[10pt]
\begin{tabular}{@{\hspace{0ex}}l l }
$= \cross{(\cross{a}, \cross{b}, \cross{c}, \cross{d}) (\ , \cross{\ },\cross{\  }, \cross{\ })}\  \cross{\cross{X}_j  \cross{\ \ }_i \cross{\ \ }_k} \ \cross{\cross{X}_i \cross{\ \ }_j\ \cross{\ \ }_k } \cross{\cross{X}_k \cross{\ \ }_i \cross{\ \ }_j } $ &($I^3J^3$) \\[10pt]
$= \cross{(\cross{a}, \cross{\ },\cross{\  }, \cross{\ })}\  \cross{\cross{X}_j \ \cross{\ \ }_i\ \cross{\ \ }_k} \ \cross{\cross{X}_i \cross{\ \ }_j\ \cross{\ \ }_k }\ \cross{\cross{X}_k \cross{\ \ }_i\ \cross{\ \ }_j } $ &(Def. \ref{def2}) \\[10pt]
$= \cross{(\cross{a}, \cross{\ },\cross{\  }, \cross{\ })}\  \cross{(\cross{c}, \cross{d}, a,b) \ (\cross{\ }, \ ,\cross{\ }, \cross{\ })} \ \cross{\cross{X}_i \cross{\ \ }_j\ \cross{\ \ }_k }\ \cross{\cross{X}_k \cross{\ \ }_i\ \cross{\ \ }_j } $ &($\cross{\ }_j$, IK) \\[10pt]
$= \cross{(\cross{a}, \cross{\ },\cross{\  }, \cross{\ })}\  \cross{(\cross{\ }, \cross{d}, \cross{\ },\cross{\ }) \ } \ \cross{\cross{X}_i \cross{\ \ }_j\ \cross{\ \ }_k }\ \cross{\cross{X}_k \cross{\ \ }_i\ \cross{\ \ }_j } $ &(Def. \ref{def2}) \\[10pt]
$= \cross{(\cross{a}, \cross{\ },\cross{\  }, \cross{\ })}\  \cross{(\cross{\ }, \cross{d}, \cross{\ },\cross{\ }) \ }  \cross{(\cross{b}, a, d, \cross{c}) (\cross{\  },  \cross{\ }, \cross{\ }, \ ) } \cross{\cross{X}_k \cross{\ \ }_i \cross{\ \ }_j } $ &($\cross{\ }_i$, JK) \\[10pt]
$= \cross{(\cross{a}, \cross{\ },\cross{\  }, \cross{\ })}\  \cross{(\cross{\ }, \cross{d}, \cross{\ },\cross{\ }) \ } \ \cross{ (\cross{\  },  \cross{\ }, \cross{\ }, \cross{c} ) }\ \cross{\cross{X}_k \cross{\ \ }_i\ \cross{\ \ }_j } $ &(Def. \ref{def2}) \\[10pt]
$= (a, \ ,\ , \ )\  (\ , d, \ ,\ )  \  (\ ,  \ , \ , c ) \ \ \cross{\cross{X}_k \cross{\ \ }_i\ \cross{\ \ }_j } $ &(Def. 1(\ref{eqcross})) \\[10pt]
$= (a, \ ,\ , \ )\  (\ , d, \ ,\ )  \  (\ ,  \ , \ , c ) \ \ \cross{(\cross{d}, c, \cross{b}, a)\  (\cross{\ }, \cross{\ }, \ , \cross{\ }) } $ &($\cross{\ }_k$, IJ) \\[10pt]
$= (a, \ ,\ , \ )\  (\ , d, \ ,\ )  \  (\ ,  \ , \ , c ) \ \ \cross{  (\cross{\ }, \cross{\ }, \cross{b} , \cross{\ }) } $ &(Def. \ref{def2}) \\[10pt]
$= (a, \ ,\ , \ )\  (\ , d, \ ,\ )  \  (\ ,  \ , \ , c ) \ \  (\ , \ , b , \ )  $ &(Def. 1(\ref{eqcross})) \\[10pt]
$= (a, d, b, c )   $ &(Def. \ref{def2}) \\[10pt]
\end{tabular}\\[5pt]

By considering these expressions, it is possible to see that more complex expressions can easily be constructed. For example, consider the conjunction of Examples 1 and 2: 
\begin{equation}
    \cross{\cross{(a,b,\cross{c},d)}\ \cross{(a,d,b,c)}}.
\end{equation}
 This simplifies within the 4-tuples as the following expressions in LoF. 
\begin{equation}
    (a, \cross{\cross{b}\cross{d}}, \cross{b}c, \cross{\cross{c}\cross{d}})
\end{equation}
which we leave as an exercise for the reader to confirm.

\bibliographystyle{ws-rv-van}
\bibliography{ws-rv}

@book{LOF1,
  author    =   "G. Spencer-Brown" ,
  title     =   "Laws of Form",
  publisher =   "Allen and Unwin Ltd.",
  address   =   "London",
  year      =   "1969"}

@book{zinfinity,
  author    =   "Peter Rowlands" ,
  title     =   "Zero to Infinity: The Foundations of Physics",
  publisher =   "World Scientific",
  address   =   "Singapore",
  year      =   "2007"}

@book{Mod1,
  author    =   "C. I. Lewis and C. R. Langford",
  title     =   "Symbolic Logic",
  publisher =   "Dover Publications, Inc.",
  address   =   "Reading, MA",
  year      =   "1969",
  edition   =   "2nd"}

@article{kau1,
    author  =   "L. H. Kauffman and F. J. Varela",
     title  =   "Form dynamics",
   journal  =   "J. soc. biol. Struct.",
  fjournal  =   "Journal for Social and
Biological Structures",
    volume  =   "3",
      year  =   "1980",
    number  =   "",
     pages  =   "171–206",
      issn  =   "0140-1750",
   mrclass  =   "",
  mrnumber  =   ""}

@book{kau2,
  author    =   "L. H. Kauffman",
  title     =   "Knots and Physics",
  publisher =   "World Scientific",
  address   =   "Singapore",
  year      =   "2001",
  edition   =   "3rd"}

@article{kau3,
    author  =   "L. H. Kauffman",
     title  =   "Majorana fermions and representations of the braid
group",
   journal  =   "Int. J. Mod. Phys. A",
  fjournal  =   "International Journal of Modern Physics A",
    volume  =   "33",
      year  =   "2018",
    number  =   "23",
     pages  =   "1-28",
      issn  =   "0217-751X",
   mrclass  =   "",
  mrnumber  =   ""}

@article{kau9,
    author  =   "L. H. Kauffman",
     title  =   "Paper Computers, Imaginary Values and the Emergence of Fermions ",
   journal  =   "Cybernetics and Human Knowing",
  fjournal  =   "Cybernetics and Human Knowing",
    volume  =   "26",
      year  =   "2019",
    number  =   "Nos. 2-3",
     pages  =   "107-160",
      issn  =   "",
   mrclass  =   "",
  mrnumber  =   ""}

@article{kau4,
    author  =   "L. H. Kauffman",
     title  =   "Self-reference and recursive forms",
   journal  =   "J. soc. biol. Struct.",
  fjournal  =   "Journal for Social and Biological Structures",
    volume  =   "10",
      year  =   "1987",
    number  =   "",
     pages  =   "53-72",
      issn  =   "0140-1750",
   mrclass  =   "",
  mrnumber  =   ""}

@inbook{kau5,
  author    =   "L. H. Kauffman",
  year      =   "1995",
  editor    =   "L.H. Kauffman",
  title     =   "Knot logic",
  booktitle =   "Knots and Applicationns",
  publisher =   "World Scientific",
  series    =   "Series on Knots and Everything",
  address   =   "Singapore",
  volume    =   "6",
  pages     =   "1-110"}

@unpublished{bfmodal,
  author    =   "A. M. Collings and L. H. Kauffman",
  title     =   "{The BF Calculus and Modal Logic }",
  note      =   "This paper expresses 16 normal modal logics within the BF Calculus",
  year      =   "2022 Unpublished"}

@inbook{BF3,
  author    =   "A. M. Collings and L. H. Kauffman",
  year      =   "2023",
  editor    =   "L. H. Kauffman and F. Cummins and R. Dibble and L. Conrad and G. Ellsbury and A. Compton and F. Grote",
  title     =   "The {BF} {C}alculus and the {S}quare {R}oof of {N}egation",
  booktitle =   "{Laws} of {F}orm: A {F}iftieth {A}nniversary",
  publisher =   "World Scientific",
  series    =   "Knots and Everything",
  volume    =   "72",
  pages     =   "253--283"}

@inbook{BF4,
  author    =   "A. M. Collings",
  year      =   "2023",
  editor    =   "L. H. Kauffman and F. Cummins and R. Dibble and L. Conrad and G. Ellsbury and A. Compton and F. Grote",
  title     =   "{T}he {BF} {C}alculus: {A} {C}omplete {F}our-{V}alued {E}xtension to {L}aws of {F}orm",
  booktitle =   "{Laws} of {F}orm: A {F}iftieth {A}nniversary",
  publisher =   "World Scientific",
  series    =   "Knots and Everything",
  volume    =   "72",
  pages     =   "199--252"}

@book{Mod5,
  author    =   "G. Boolos",
  title     =   "The Logic of Provabilty",
  publisher =   "Cambridge University
Press",
  address   =   "Cambridge",
  year      =   "1993",
  edition   =   ""}

@article{BILx6,
    author  =   "Y. Shramko and H. Wansing",
     title  =   "Some useful 16-valued logics: How a
computer network should think",
   journal  =   "J. Philos. Log.",
  fjournal  =   "Journal of Philosophical Logic",
    volume  =   "34",
      year  =   "2005",
    number  =   "",
     pages  =   "121-153",
      issn  =   "0022-3611",
   mrclass  =   "",
  mrnumber  =   ""}

@article{Braid1,
    author  =   "E. Artin",
     title  =   "Theory of Braids",
   journal  =   "Ann. Math.",
  fjournal  =   "Annals of Mathematics",
    volume  =   "48",
      year  =   "1947",
    number  =   "",
     pages  =   "101-126",
      issn  =   "0003-486X",
   mrclass  =   "",
  mrnumber  =   ""}

@book{Braid2,
  author    =   "J. Birman",
  title     =   "Braids, links and mapping class groups",
  publisher =   "Princeton University Press, Series: Annals of Mathematics Studies.",
  address   =   "Princeton",
  year      =   "1975",
  edition   =   ""}

@article{Braid3,
    author  =   "B. Schmeikal",
     title  =   "Four Forms Make a Universe",
   journal  =   "Adv. Appl. Clifford Algebr.",
  fjournal  =   "Advances in Applied
Clifford Algebras",
    volume  =   "26",
      year  =   "2016",
    number  =   "",
     pages  =   "889-911",
      issn  =   "0003-486X",
   mrclass  =   "",
  mrnumber  =   ""}

@article{BILx2,
    author  =   "M. Ginsburg",
     title  =   "Multivalued logics: a uniform approach to inference in
artificial intelligence",
   journal  =   "Comput. Intell.",
  fjournal  =   "Computational Intelligence",
    volume  =   "4",
      year  =   "1988",
    number  =   "3-4",
     pages  =   "265-316",
      issn  =   "0824-7935",
   mrclass  =   "",
  mrnumber  =   ""}

@inbook{BILx1,
  author    =   "N. Belnap",
  year      =   "1977",
  editor    =   "J. M. Dunn and G. Epstein",
  title     =   "A useful four-valued logic",
  booktitle =   "Modern Uses of Multiple-Valued Logic",
  publisher =   "D. Reidel",
  series    =   "",
  address   =   "Dordrecht",
  volume    =   "",
  pages     =   "7-37"}

@inbook{BILx4,
  author    =   "M. Fitting",
  year      =   "2006",
  editor    =   "T. Borland and S. A. Pederson",
  title     =   "Bilattices are nice things",
  booktitle =   "Self-Reference",
  publisher =   "Stanford Center for the Study of Language and Information",
  series    =   "",
  address   =   "Stanford",
  volume    =   "",
  pages     =   "53-77"}


\end{document}